\documentclass[runningheads]{llncs}
\usepackage{graphicx}
\usepackage{xspace} 
\usepackage{amssymb,amsmath}
\usepackage{comment}
\usepackage{subcaption} 
\usepackage[perpage]{footmisc}
\usepackage[margin=1.1in]{geometry}

\makeatletter
\let\@@citation@@=\citation
\renewcommand{\citation}[1]{\@@citation@@{#1}%
	\@for\@tempa:=#1\do{\@ifundefined{cit@\@tempa}%
		{\global\@namedef{cit@\@tempa}{}}{}}%
}
\makeatother

\usepackage{hyperref}

\makeatletter
\def\@lbibitem[#1]#2#3\par{%
	\@ifundefined{cit@#2}{}{\@skiphyperreftrue
		\H@item[%
		\ifx\Hy@raisedlink\@empty
		\hyper@anchorstart{cite.#2\@extra@b@citeb}%
		\@BIBLABEL{#1}%
		\hyper@anchorend
		\else
		\Hy@raisedlink{%
			\hyper@anchorstart{cite.#2\@extra@b@citeb}\hyper@anchorend
		}%
		\@BIBLABEL{#1}%
		\fi
		\hfill
		]%
		\@skiphyperreffalse}%
	\if@filesw
	\begingroup
	\let\protect\noexpand
	\immediate\write\@auxout{%
		\string\bibcite{#2}{#1}%
	}%
	\endgroup
	\fi
	\ignorespaces
	\@ifundefined{cit@#2}{}{#3}}

\def\@bibitem#1#2\par{%
	\@ifundefined{cit@#1}{}{\@skiphyperreftrue\H@item\@skiphyperreffalse
		\Hy@raisedlink{%
			\hyper@anchorstart{cite.#1\@extra@b@citeb}\relax\hyper@anchorend
		}}%
		\if@filesw
		\begingroup
		\let\protect\noexpand
		\immediate\write\@auxout{%
			\string\bibcite{#1}{\the\value{\@listctr}}%
		}%
		\endgroup
		\fi
		\ignorespaces
		\@ifundefined{cit@#1}{}{#2}}
	\makeatother
	
\makeatletter
\renewcommand{\@Opargbegintheorem}[4]{%
	#4\trivlist\item[\hskip\labelsep{#3#2\@thmcounterend}]}
\makeatother

\spnewtheorem*{repeatthm}{}{\bfseries\upshape}{\itshape}

\def\mybreak{\hfill\mbox{}\break\\}
\newcommand{\qedwhite}{\hfill \ensuremath{\Box}}

\def\F{\mbox{\ensuremath{\mathcal F}}\xspace}
\def\HH{\mbox{\ensuremath{\mathcal H}}\xspace}

\def\G{\mbox{\ensuremath{\mathcal G}}\xspace}

\def\p{\mbox{\ensuremath{\mathbf p}}\xspace}
\def\q{\mbox{\ensuremath{\mathbf q}}\xspace}
\def\x{\mbox{\ensuremath{\mathbf x}}\xspace}
\def\y{\mbox{\ensuremath{\mathbf y}}\xspace}
\def\z{\mbox{\ensuremath{\mathbf z}}\xspace}
\def\r{\mbox{\ensuremath{\mathbf r}}\xspace}
\def\Rd{\mbox{\ensuremath{\mathbb{R}^d}}\xspace}
\def\CH{\mbox{\ensuremath{\mathcal{CH}}}\xspace}

\mathcode`l="8000
\begingroup
\makeatletter
\lccode`\~=`\l
\DeclareMathSymbol{\lsb@l}{\mathalpha}{letters}{`l}
\lowercase{\gdef~{\ifnum\the\mathgroup=\m@ne \ell \else \lsb@l \fi}}%
\endgroup

\begin{document}
\title{Coloring Delaunay-Edges and their Generalizations\thanks{Research by the first author was partially supported by ERC AdG Disconv and MTA EU10/2016-11001. Research by the second author was supported by the National Research, Development and Innovation Office -- NKFIH under the grant K 116769 and K 132696. Research by the second and third authors was supported by the Lend\"ulet program of the Hungarian Academy of Sciences (MTA), under grant number LP2017-19/2017.}}
%
%
\author{Eyal Ackerman\inst{1} \and
Bal\'azs Keszegh\inst{2,3} \and
D\"om\"ot\"or P\'alv\"olgyi\inst{3}}
\authorrunning{E. Ackerman et al.}
%
\institute{Department of Mathematics, Physics, and Computer Science,
	University of Haifa at Oranim, 	Tivon 36006, Israel. E-mail: \email{ackerman@sci.haifa.ac.il}
	\and
Alfr\'ed R{\'e}nyi Institute of Mathematics, Hungarian Academy of Sciences Re\'altanoda u. 13-15 Budapest, 1053 Hungary. E-mail: \email{keszegh@renyi.hu}
\and
MTA-ELTE Lend\"ulet Combinatorial Geometry Research Group, Institute of Mathematics, E\"otv\"os Lor\'and University (ELTE), Budapest, Hungary.}
\maketitle              
\begin{abstract}
	We consider geometric hypergraphs whose vertex set is a finite set of points (e.g., in the plane),
	and whose hyperedges are the intersections of this set with a family of geometric regions (e.g., axis-parallel rectangles).
	A typical coloring problem for such geometric hypergraphs asks, given an integer $k$, for the existence of an integer $m=m(k)$, such that every set of points can be $k$-colored such that every hyperedge of size at least $m$ contains points of different (or all $k$) colors.
	
	We generalize this notion by introducing coloring of \emph{$t$-subsets} of points such that every hyperedge that contains enough points contains $t$-subsets of different (or all) colors.
	In particular, we consider all $t$-subsets and $t$-subsets that are themselves hyperedges.
	The latter, with $t=2$, is equivalent to coloring the edges of the so-called \emph{Delaunay-graph}.
	
	In this paper we study colorings of Delaunay-edges with respect to halfplanes, pseudo-disks, axis-parallel and bottomless rectangles, and also discuss colorings of $t$-subsets of geometric and abstract hypergraphs, and connections between the standard coloring of vertices and coloring of $t$-subsets of vertices.

\end{abstract}
\section{Introduction}

According to the classic Ramsey Theorem, for every integer $m$ there exists an integer $n=n(m)$ such that for every $2$-coloring of the edges of the complete graph $K_n$ there exists a monochromatic copy of $K_m$.
Moreover, an extension of Ramsey Theorem states that for all integers $k$, $t$ and $m$ there exists an integer $n=n(k,t,m)$ such that for every $k$-coloring of all the hyperedges of the $t$-uniform complete hypergraph\footnote{A \emph{hypergraph} is a pair $(\cal V,\cal E)$ where $\cal V$ is a nonempty set (\emph{vertices}) and $\cal E$ is a set of subsets of $\cal V$ (\emph{hyperedges}). A hypergraph is \emph{$t$-uniform} if $|e|=t$ for every $e \in \cal E$. A $t$-uniform hypergraph is \emph{complete} if every $t$-subset of $\cal V$ is a hyperedge.} on $n$ vertices there exists a monochromatic copy of the $t$-uniform complete hypergraph on $m$ vertices.
Obviously, one wishes to find the minimum integer $n$ that satisfies this property, however, only asymptotic bounds on $n$ are known~\cite{RamseyHypergraph}.

From a different perspective, given $k$, $t$ and $n$, one can look for the largest integer $m=m(k,t,n)$ such that for every $k$-coloring of the complete $t$-uniform hypergraph on $n$ vertices there is a subset of $m-1$ vertices such that each $t$-subset of them has the same color.
Or, equivalently, the smallest integer $m=m(k,t,n)$ such that there exists a $k$-coloring of the complete $t$-uniform hypergraph on $n$ vertices for which every subset of $m$ vertices contains hyperedges of different colors. 
A coloring that satisfies the latter is called a \emph{proper} coloring.
If every subset of $m$ vertices contains a hyperedge of each of the $k$ colors, then we say that the coloring is \emph{polychromatic}.

\medskip
We consider such colorings of hyperedges in geometric settings.
Specifically, the vertex set is a finite set of points $S \subseteq \mathbb{R}^d$ and we consider only $m$-subsets of $S$ that are induced by a family of geometric regions. For example, subsets $S' \subseteq S$ such that there is an axis-parallel $d$-dimensional box $B$ with $S'=S \cap B$.

Since the geometric setting is quite restrictive, many times $m$ does not depend on $n$.
For example, we prove such a result for subsets that are induced by \emph{$\HH$-regions}.
Given a finite family of halfspaces $\HH=\{H_1,\dots, H_h\}$ we call a region $R$ an \emph{$\HH$-region} if it is the intersection of finitely many halfspaces, such that each of them is a translate of one of the halfspaces in $\HH$.

\begin{theorem}\label{thm:region}
	For all integers $t\ge 2$, $k \ge 1,h\ge 1$ there exists an integer $m=m(k,t,h)$ with the following property: for every dimension $d \ge 1$ and a finite family of $h$ halfspaces $\HH$ in $\mathbb{R}^d$, the $t$-subsets of every finite set of points $S$ in $\mathbb{R}^d$ can be $k$-colored such that every $\HH$-region that contains at least $m$ points from $S$ contains a $t$-subset of points of each of the $k$ colors. 
\end{theorem}

As a corollary, similar results hold for subsets that are induced by $d$-dimensional axis-parallel boxes and by homothets\footnote{A homothetic copy of a set is a scaled and translated copy of the set (rotations are not allowed).}  of a $d$-dimensional convex polytope.

\medskip

It is natural to consider not only $m$-subsets that are induced by geometric constraints, but also hyperedges. Given a set of points $S$ and a family of regions $\F$ (where a \emph{region} stands for a set of points, e.g., a disk or halfplane in the plane), we define the \emph{Delaunay-hypergraph} $G(S,\F)$ as follows: $S$ is its vertex set and a subset $S'\subseteq S$ is a (Delaunay-)hyperedge if there exists a region $F\in \F$ such that $S'=S\cap F$.
A Delaunay-hyperedge of size $t$ is called a $t$-Delaunay-hyperedge.

This definition gives rise to the following type of questions whose study we initiate here:
\begin{quote}
	Given a family of regions $\cal F$ in $\mathbb{R}^d$ and integers $k$ and $t$, is there an integer $m=m_{\cal F}(k,t)$ such that the following holds. For every set of points $S \subseteq \mathbb{R}^d$ there exists a $k$-coloring of the $t$-Delaunay-hyperedges of $G(S,\F)$ such that every subset of at least $m$ points that is induced by a region from $\F$ contains $t$-Delaunay-hyperedges of two different colors.
\end{quote}

If polychromatic coloring is sought, then `two' is replaced by $k$.
We remark that these questions have been studied extensively for $t=1$ and various types of geometric regions, see the related work section later on.
Therefore, our work continues and generalizes this line of research.

\smallskip
For convenience, instead of considering subsets of points that are induced by a region of $\F$, we will sometimes consider (equivalently) Delaunay-hyperedges.
Furthermore, since in some cases an $m$-Delaunay-hyperedge might contain very few or no $t$-Delaunay-hyperedges (for $1<t<m$), we sometimes have to consider only $m$-Delaunay-hyperedges that contain enough $t$-Delaunay-hyperedges.

\smallskip
Especially interesting is the case in which we only consider Delaunay-hyperedges of size $t=2$, called \emph{Delaunay-edges}. In this case $G$ is the so-called \emph{Delaunay-graph} of $S$ with respect to $\F$. 
See Figure \ref{fig:hpexample} for an illustration of these notions and various types of colorings that we consider.

\begin{figure}
	\centering
	\includegraphics[width=12cm]{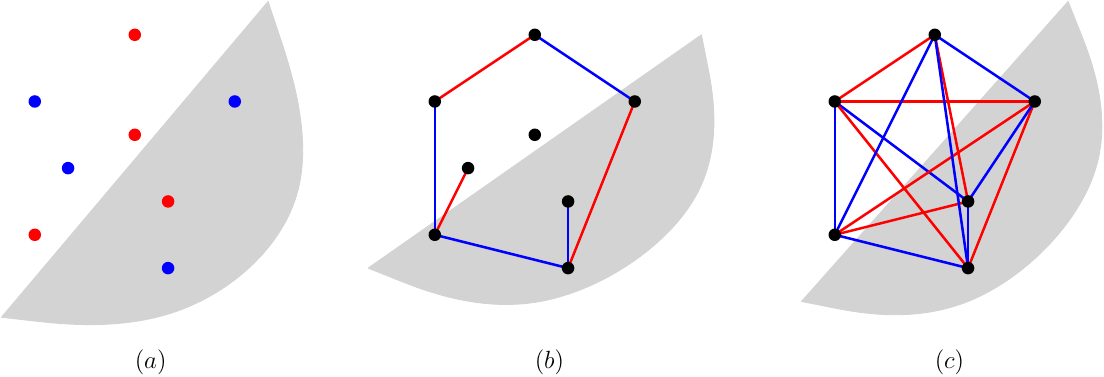}
	\caption{Colorings with respect to halfplanes: a) coloring the points ($t=1$) such that any $3$-Delaunay-hyperedge is non-monochromatic; b) coloring the Delaunay-edges ($t=2$) such that any Delaunay-hyperedge containing three Delaunay-edges is not monochromatic; c) coloring all $2$-subsets (not necessarily Delaunay-hyperedges) such that any $3$-Delaunay-hyperedge is non-monochromatic. (Two points were removed for better visibility.)}
	\label{fig:hpexample}
\end{figure}

Note that the notion of Delaunay-graph generalizes the well-known concept of the Delaunay \emph{triangulation} of a planar point set $S$ (see, e.g.,~\cite{MarksBook}), which, using the above notation, is equivalent to the Delaunay-graph of $S$ with respect to the family of all disks (assuming no disk contains four points of $S$ on its boundary).

\smallskip
In order to state our main result in the context of coloring Delaunay-hyperedges 
we need the following definitions.
A family of compact regions in the plane, each of which is bounded by a Jordan curve, is called a family of \emph{pseudo-disks} if every two boundary curves are disjoint or properly cross twice.

\begin{definition}
	We say that a family of regions $\F$ is \emph{shrinkable} with respect to a set of points $S$ if for every region $F\in \F$ with $|F\cap S|\ge 3$ and every $p\in F\cap S$ there exists another region $F'\in \F$ such that $p\in F'\cap S\subset F\cap S$ and $|F'\cap S|=|F\cap S|-1$. 
\end{definition}

Note that we do not require $F' \subseteq F$. 
For example, disks are shrinkable with respect to any finite point set whereas unit disks are not. 

\begin{theorem}\label{thm:pseudo}
	Let $S$ be a finite set of points in the plane, let $\F$ be a finite family of pseudo-disks that is \emph{shrinkable} with respect to $S$. Then it is possible to color the Delaunay-edges of $S$ with respect to $\F$ with four colors, such that every pseudo-disk $F \in \F$ that contains at least three points from $S$ contains Delaunay-edges of two different colors.
\end{theorem}

Theorem~\ref{thm:pseudo} also applies to the case where $\F$ is the family of all homothets of a (closed) polygon $P$ and $S$ is in a certain general position with respect to $P$ (namely, no homothet of $P$ contains four points of $S$ on its boundary, and no two points in $S$ define a line that is parallel to a line through two vertices of $P$). Indeed, one can apply Theorem~\ref{thm:pseudo} on such a family since: it trivially holds that there is a finite subfamily which defines the same Delaunay-hypergraph on $S$; if needed, a slight perturbation of such a finite subfamily can turn it into a family of pseudo-disks without changing the Delaunay-hypergraph; 
and it can be shown that such a family is shrinkable (see~\cite[Lemma~2.3]{AKV15}). 

\medskip
Less than four colors suffice for two specific families.

\begin{theorem}\label{thm:halfplanes}
	Let $S$ be a finite set of points in the plane and let $\F$ be the family of all halfplanes. Then the Delaunay-edges of $S$ with respect to $\F$ can be $2$-colored such that every halfplane that contains at least three Delaunay-edges contains two differently colored edges. This statement is false if we replace three with two.
\end{theorem}

We remark that the family of all halfplanes is not always shrinkable.
For example, consider this family with respect to six points that consists of the vertices of a regular pentagon and its center. If $F$ is a halfplane that contains two vertices of the pentagon and its center $p$, then there is no halfplane $F'$ that contains just $p$ and exactly one of the other two points in $F$.
For this reason a halfplane may contain an arbitrary number of points from $S$ while containing only one Delaunay-edge; this is why we consider halfplanes that contain enough Delaunay-edges instead of halfplanes that contain enough points.

\smallskip
An \emph{axis-parallel rectangle} is a rectangle in the plane whose sides are parallel to the $x$- and $y$-axes. A region is a \emph{bottomless rectangle} if it is the homothetic copy of the set $\{(x,y):0<x<1,y<0\}$). Theorem~\ref{thm:pseudo} can be applied to the family of bottomless rectangles, however, two colors already suffice in this case: 

\begin{theorem} \label{thm:bless}
	Let $S$ be a finite set of points in the plane and let $\F$ be the family of all bottomless rectangles. Then the Delaunay-edges of $S$ with respect to $\F$  can be $2$-colored such that every bottomless rectangle that contains at least four points from $S$ contains two differently colored edges. This statement is false if we replace four with three.
\end{theorem}

We also consider coloring Delaunay-edges with respect to axis-parallel rectangles. However, in this case we only managed to show that by using $O(\log |S|)$ colors one can guarantee that every rectangle that contains at least three points from $S$ contains two differently colored Delaunay-edges. 
We leave it as an open problem whether constantly many colors also suffice.

\medskip
Suppose that for a given family of regions, $k$ and $t$ we have $m=m(k,t)$ such that there is a polychromatic $k$-coloring of $t$-Delaunay-hyperedges (resp., $t$-subsets) such that every set of at least $m$ vertices induced by a region from the family contains a hyperedge (resp., subset) of each color. Then it seems reasonable to expect that a similar statement holds for every $t' > t$ (with a different $m$). We could prove such a statement for subsets induced by $\HH$-regions and coloring of all $t$-subsets (not necessarily Delaunay-hyperedges).

\begin{theorem}\label{thm:t-ind}	
		Given a family of $h$ halfspaces $\HH$  in $\mathbb{R}^d$ and a positive integer $k$, suppose that for some positive integer $t$ there exists an integer $m=m(k,t,\HH)$ with the following property: the $t$-subsets of every finite point set $S$ in $\mathbb{R}^d$ can be $k$-colored such that every $\HH$-region that contains at least $m$ points from $S$ contains a $t$-subset of each of the $k$ colors. Then the same property holds when exchanging $t$ with an integer $t' > t$ and $m$ with $m'=m+t'-t$.	
\end{theorem}

We could not prove a similar statement when instead of all $t$-subsets we only color $t$-Delaunay-hyperedges.
A possible reason that obtaining such a result is harder, is that the number of $t$-Delaunay-hyperedges increases more moderately when compared to the number of all $t$-subsets. For example, in the case of a pseudo-disk family, a pseudo-disk that contains $m$ points contains at most $O(t^2m)$ Delaunay-hyperedges of size at most $t$~\cite{buzaglo}, whereas it contains ${m \choose t}$ $t$-subsets (that are not necessarily Delaunay-hyperedges).
Nevertheless, we believe that Theorem \ref{thm:t-ind} (and also Theorem \ref{thm:region}) should also hold for $t$-Delaunay-hyperedges.

It is also an open question whether a similar statement holds in the abstract (non-geometric) setting, that is, for every subset of $m$ vertices.
We could only prove such a statement in the abstract setting for $t=1$.
We also remark that it is an open problem whether proper $2$-colorability of vertices ($t=1$) implies polychromatic $k$-colorability of vertices for $k>2$; see~\cite{PPT14}.

\smallskip
In general, call an abstract hypergraph family \emph{closed for subhypergraphs} if for every  hypergraph $(\cal V,\cal E)$ in the family each of its subhypergraphs $(\cal V',\cal E')$ where $\cal V'\subset \cal V$ and ${\cal E'}=\{e \cap V'\mid e\in \cal E\}$, is also in the family.
Denote by $(k,t)\to (k',t')$ the following statement:
IF for an abstract hypergraph family closed for subhypergraphs there is an integer $m$ such that for any hypergraph from the family the $t$-subsets of its vertices can be $k$-colored such that any hyperedge with at least $m$ vertices contains a $t$-subset of each of the $k$ colors,
THEN there is an integer $m'$ such that for any hypergraph from the family the $t'$-subsets of its vertices can be $k'$-colored such that any hyperedge with at least $m'$ vertices contains a $t'$-subset of each of the $k'$ colors.
Note that the implications $(k,t)\to (k',t)$ trivially hold for any $k'\le k$ and a central question of polychromatic $k$-coloring is whether $(2,1)\to (k',1)$ holds for every $k'$.
We show in Proposition~\ref{prop:chi} that $(k,t)\to (2,t')$ for every $k\ge 2$ and $t'> t$ and in Proposition \ref{prop:mk} that $(k,1)\to (k',t')$ for $k'\le \sum_{i=0}^{t'} \binom{k-1}i$.
In fact, in the proof of these statements we do not use that the family is closed for subhypergraphs.
The general question whether $(k,t)\to (k',t')$ holds for every $k,k'$ and $t'\ge t$ remains open.

\paragraph{Organization.} We consider coloring Delaunay-edges in Section~\ref{sec:deledge}:
Theorems~\ref{thm:pseudo}, \ref{thm:halfplanes} and \ref{thm:bless} are proved in  Sections~\ref{sec:psdisk}, \ref{sec:hp} and \ref{sec:bless}, respectively.
In Sections~\ref{sec:allthyp} and \ref{sec:alltgeom} we consider coloring of all the $t$-subsets: first, in Section~\ref{sec:allthyp} we prove results of the form $(k,t)\to (k',t')$ for abstract hypergraphs. Then, in Section~\ref{sec:alltgeom} we go back to the geometric setting and prove such a result for $\HH$-regions,  Theorem~\ref{thm:t-ind}, and use it to prove Theorem~\ref{thm:region}.

\subsection{Related work}

The notion of coloring the edges of a graph is as old as that of coloring the vertices of a graph, where often these colorings exhibit substantially different behaviors. Edge-coloring problems give rise to completely new questions, as in the case of Ramsey theory, while at other times they are natural extensions of problems about vertex-colorings.  
For example, in the paper of Erd\H os et al.~\cite{E71} Euclidean Ramsey theorems concerning inevitable monochromatic structures when coloring the points in $\mathbb{R}^d$ are studied. In the last section of this paper, the authors prove Ramsey-type results for coloring all the \emph{pairs} of points ($2$-subsets). 
For example, they prove that it is possible to $2$-color every pair of points in the plane, such that the vertices of every acute triangle define pairs of both colors.
This generalization from coloring of points to coloring of pairs of points is of the same flavor as our coloring of $t$-subsets.

Another generalization of this flavor is a recent work of Mustafa and Ray~\cite{MR17} concerning \emph{$\epsilon$-Mnets}.
Given a hypergraph (set system) $(\cal V, \cal E)$ and a constant $\epsilon > 0$, an \emph{$\epsilon$-net} is a subset $N \subseteq \cal V$ such that for every $e \in \cal E$ with $|e| \ge \epsilon|\cal V|$ we have $N \cap e \ne \emptyset$.
Epsilon-nets are well-studied due to their applications in discrete and computational geometry, approximation algorithms and machine learning.
Mustafa and Ray~\cite{MR17} introduced the notion of \emph{$\epsilon$-Mnets}:
An $\epsilon$-Mnet is a family of subsets $\{M_1,M_2,\ldots,M_l\}$ each of size $\Omega(\epsilon|\cal V|)$, such that for every $e \in \cal E$ with $|e| \ge \epsilon|\cal V|$ we have $M_i \subseteq e$ for some $1 \le i \le l$.
Thus, $\epsilon$-Mnets generalize $\epsilon$-nets in a way which is similar to our generalization of vertex-coloring to coloring of $t$-hyperedges and $t$-subsets.
Moreover, suppose that 
there are constants $c_1,c_2 > 0$ and a $k$-coloring of the $t$-hyperedges of a certain hypergraph on $n$ vertices such that every induced subset of $m$ vertices contains a $t$-hyperedge of each of the $k$ colors, for $t=c_1 \epsilon n$ and $m = c_2 \epsilon n$ ($c_2 > c_1$). Then each of the $k$ colors is an $\epsilon$-Mnet of the hypergraph.

\medskip
We are not aware of any work that considers coloring of Delaunay-hyperedges or $t$-subsets of vertices as defined in this paper. There is however a vast literature on coloring points with respect to Delaunay-hyperedges (that is, when $t=1$). Most of the results in this area are about the same planar families of regions that we study in this paper, namely, pseudo-disks, halfplanes and axis-parallel rectangles. One usually considers the existence of proper two-colorings with respect to these families. 
This has motivated us to investigate the existence of proper two-colorings of Delaunay-edges (that is, when $t=2$) with respect to the same planar families.
Here we briefly describe the results that are most relevant to our work.

It is known that one can properly $4$-color points with respect to pseudo-disks since their Delaunay-graph is planar~\cite{buzaglo}. On the other hand, there is no constant $m$ for which we can properly $2$-color points with respect to (unit or pseudo) disks containing at least $m$ points~\cite{PP13,PTT05}.
That is, in our terminology, $m_{\F}(2,1)$ does not exist when $\F$ is the family of (unit) disks in the plane. It is an open problem whether $m_{\F}(3,1)$ exists, even for unit disks~\cite{K11,KP16}. 

For halfplanes containing $3$ points there is a $2$-coloring~\cite{K11}, and more generally, for halfplanes containing $2k-1$ points there is a polychromatic $k$-coloring~\cite{SY10}.
It is also known that there is a proper $2$-coloring with respect to bottomless rectangles containing $4$ points \cite{K11}, and more generally, for bottomless rectangles containing $3k-2$ points there is a polychromatic $k$-coloring~\cite{A13}. 

Considering the family of all axis-parallel rectangles, observe that the Delaunay-graph in this case is not necessarily planar. Estimating the chromatic number of this Delaunay-graph is a famous open problem with a large gap between the best known lower bound of $\Omega(\frac{\log n}{\log\log n})$~\cite{chenpach} and the best known upper bound of $O(n^{0.368})$~\cite{chan12} (where $n$ is the number of points).
There is however a proper $O(\log n)$-coloring with respect to axis-parallel rectangles that contain at least $3$ points from the given point set~\cite{eyalrom}. 

For results about polychromatic coloring of vertices of abstract hypergraphs see, e.g., \cite{BPRS13,TX}.
We remark that in the case of \emph{hereditary} hypergraph families, it is an open problem whether $2$-colorability (with respect to hyperedges of size at least $m$, for some $m$) implies polychromatic $k$-colorability for $k > 2$ (with respect to hyperedges of size at least $m_k$, for some $m_k$ that depends on $m$ and $k$); see \cite{PPT14}.
The only related result that we are aware of is by Berge \cite{B89} who showed that $m_2=2$ implies $m_k=k$ for all $k$.
For more results, see the survey \cite{PPT14} or the webpage \cite{cogezoo}.

\smallskip
The geometric hypergraphs that we consider are usually called \emph{primal} hypergraphs. We briefly mention that one can define \emph{dual} hypergraphs where the vertices correspond to regions from $\F$ and the hyperedges correspond to points of the plane (a hyperedge that corresponds to the point $p$ consists of every vertex that corresponds to a region $F\in \F$ that contains $p$). 
One can even generalize these notions further, in particular one can define \emph{intersection} hypergraphs where the vertices correspond to regions from a family $\F$ and the hyperedges correspond to regions from a family $\G$ (the hyperedge that corresponds to a region $G\in \G$ contains exactly those vertices whose corresponding region $F\in F$ intersects $G$). All of these variants have a vast literature with connections to the notions of \emph{cover-decomposability} and \emph{conflict-free} coloring.

\section{Coloring Delaunay-edges}\label{sec:deledge}

Let $S$ be a finite set of points in the plane and let $\F$ be a family of regions. Recall that $G(S,\F)$ is the Delaunay-hypergraph whose vertex set is $S$ and whose hyperedge set consists of all subsets $S' \subseteq S$ for which there exists a region $F \in \F$ such that $S' = S \cap F$.

\begin{definition}
	The \emph{Delaunay-edges} of $S$ with respect to $\F$ are all the hyperedges of size two in $G(S,\F)$.
	The \emph{Delaunay-graph} $D(S,\F)$ is the graph whose vertex set is $S$ and whose edge set consists of the Delaunay-edges of $S$ with respect to $\F$.
\end{definition}

Our goal is to find a coloring such that every region in $\cal F$ that contains $m=m(\F)$ \emph{points} from $S$ contains two differently colored Delaunay-edges.
The families that we consider are pseudo-disks, axis-parallel rectangles, bottomless rectangles and halfplanes. 

\subsection{Coloring Delaunay-edges with respect to pseudo-disks}\label{sec:psdisk}

We say that a family of pseudo-disks $\F$ is \emph{saturated} with respect to a set of points $S$, if there is no region $R$ such that $\F \cup \{R\}$ is a family of pseudo-disks and $G(S,\F) \neq G(S,\F \cup \{R\})$. 
Note that every family of pseudo-disks can be extended greedily to a saturated family with respect to $S$.

\subsubsection{The Delaunay-graph with respect to pseudo-disks}

\mybreak
Let $S$ be a set of points in the plane and let $\F$ be a family of regions.
It is well-known~\cite{Kvoronoi} that when $\F$ is the set of homothets of a convex shape (or any other convex pseudo-disk family), 
then embedding the Delaunay-graph $D(S,\F)$ such that every two adjacent points in $S$ are connected by a straight-line segment results in a plane graph.
For Delaunay-graphs with respect to (non-convex) pseudo-disks this might not hold, but using the Hanani-Tutte theorem it is possible to show that the Delaunay-graph is always planar \cite{buzaglo}.
We need the following stronger result, which is proved in a companion paper~\cite{PD-plane}.

\begin{theorem}[\cite{PD-plane}]\label{thm:planedelgraph}
	Given a finite point set $S$ and a finite pseudo-disk family $\F$, there is a drawing of the Delaunay-graph $D(S,\F)$ such that every edge lies inside every pseudo-disk that contains both of its endpoints and no two edges cross.
\end{theorem}

A useful ingredient of the proof of Theorem \ref{thm:planedelgraph} is the following lemma of Buzaglo et al.~\cite{buzaglo}.

\begin{lemma}[\cite{buzaglo}]\label{lem:buzaglo}
	Let $D_0$ and $D_1$ be two pseudo-disks from a pseudo-disk family and let $\gamma_0$ and $\gamma_1$ be two curves such that $D_i$ contains $\gamma_i$ and does not contain the endpoints of $\gamma_{1-i}$, for $i=0,1$. Then $\gamma_0$ and $\gamma_1$ cross each other an even number of times.
\end{lemma}

We will use this lemma as well as the next one due to Pinchasi~\cite{Pin14}, whose proof applies a result of Snoeyink and Hershberger \cite{SH89}.

\begin{lemma}[\cite{Pin14}]\label{lem:shrink}
	Let $\cal B$ be a family of pairwise disjoint sets in the plane and let $\F$ be a family of pseudo-disks. Let $D$ be a member of $\F$ and suppose that for some integer $k\ge 2$, $D$ intersects exactly $k$ members of $\cal B$, one of which is $B \in \cal B$. Then for every integer $2 \le \ell \le k$ there exists $D' \subseteq D$ such that $D'$ intersects $B$ and exactly $\ell-1$ other sets from $\cal B$, and $\F \cup \{D'\}$ is a family of pseudo-disks.
\end{lemma}

As usual, in Lemma \ref{lem:shrink} two sets are defined to intersect if and only if they have a non-empty intersection. By setting $\cal B$ to be $S$, Lemma \ref{lem:shrink} implies that a family of pseudo-disks which is saturated with respect to $S$ is also shrinkable with respect to $S$.

\subsubsection{Coloring Delaunay-edges with respect to pseudo-disks }

\mybreak
Let $S$ be a finite set of points in the plane, let $\F$ be a finite family of pseudo-disks that is shrinkable with respect to $S$. Theorem~\ref{thm:pseudo} states that it is possible to color the edges of the Delaunay-graph $D(S,\F)$ with four colors, such that every pseudo-disk $F \in \F$ that contains at least three points from $S$ contains Delaunay-edges of two different colors.

\begin{proof}[Proof of Theorem~\ref{thm:pseudo}]
	Let $D$ be a plane drawing of $D(S,\F)$ such that every Delaunay-edge lies inside every pseudo-disk that contains its endpoints.
	By Theorem~\ref{thm:planedelgraph} such a drawing exists.
	
	Let $J$ be a subgraph of the line graph of $D(S,\F)$ that is defined as follows.
	The vertex set of $J$ is the set of Delaunay-edges, and two Delaunay-edges $e$ and $e'$ are adjacent in $J$ if they share an endpoint and there is a pseudo-disk $F \in \F$ such that $F\cap S$ is exactly the three endpoints of $e$ and $e'$. 
	
	Using $D$ we draw $J$ as follows (we do not distinguish between points and edges and their embeddings).
	For every Delaunay-edge $e$ we choose arbitrarily a point $p_e$ in the interior of this edge.
	Suppose that $e$ and $e'$ are two Delaunay-edges that are adjacent in $J$, 
	and let $q \in S$ be the common endpoint of $e$ and $e'$.
	Then the drawing of the edge $\{e,e'\}$ of $J$ consists of the segment of $e$ between $p_e$ and $q$ and the segment of $e'$ between $q$ and $p_{e'}$.
	
	Note that the drawing of $J$ does not contain crossing edges, since $D$ is a plane graph.
	However, $J$ contains several overlaps between the drawn edges.

	\begin{proposition}\label{prop:Jplanar}
	The graph $J$ can be redrawn as a plane graph.
	\end{proposition}

		\begin{proof}
			In order to remove overlaps between the edges we re-draw the edges with the following procedure (see Figure~\ref{fig:redraw} for an illustration).
			
			\begin{figure}
				\centering
				\includegraphics[width=7cm]{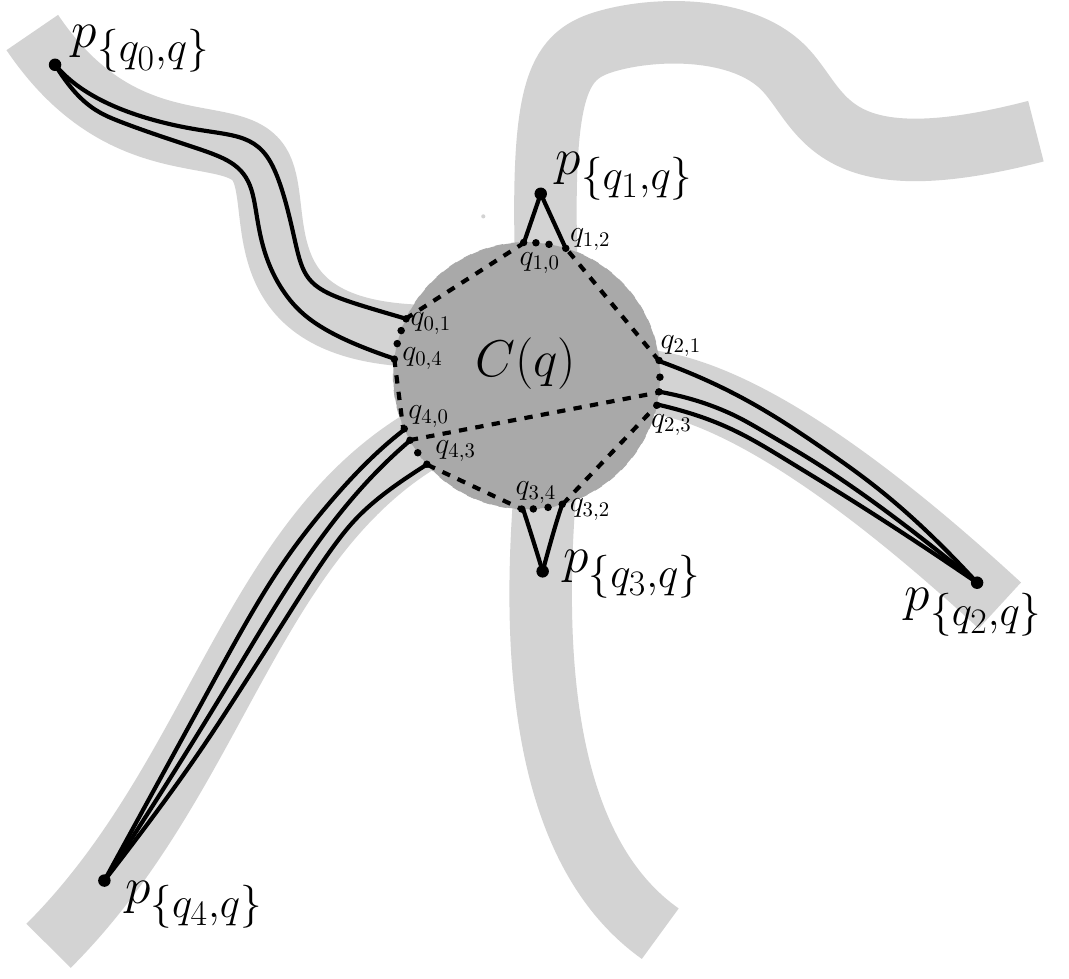}
				\caption{Drawing $J$ in the proof of Theorem~\ref{thm:pseudo}.}
				\label{fig:redraw}
			\end{figure}
			
			First, we replace every vertex $q$ by a very small disk $C(q)$ whose center is $q$. The radius of $C(q)$ should be small enough such that (1) $C(q)$ does not contain points of $S$ besides $q$, (2) the boundary $\partial C(q)$ of $C(q)$ does not intersect the boundary of a pseudo-disk from $\cal F$ (that is, $C(q)$ lies inside every pseudo-disk that contains $q$) and (3) $C(q)$ does not intersect any edge which is not adjacent to $q$ nor any $C(p)$ with $p\ne q$.  
			
			Now for each vertex $q$ and incident edge $\{p,q\}$ we take the intersection $r_{pq}$ of $\{p,q\}$ with $\partial C(q)$ which is farthest from $q$ along $\{p,q\}$ and delete the edge part that is between $r_{pq}$ and $q$. We get a drawing in which pairwise disjoint disks around the points are connected by pairwise disjoint edges that have one endpoint on the boundary of one disk and the other endpoint on the boundary of another. Additionally, besides these endpoints, the disks and the edges are  pairwise disjoint.
			
			Then we slightly inflate every (truncated) edge $\{p,q\}$ such that it has a very small width and it is still connecting $C(p)$ and $C(q)$. We make sure that the width of every edge is small enough such that the thickened edges are still pairwise disjoint. We do the thickening such that the intersection of a thickened edge and a disk that corresponds to one of its endpoints consists of an arc on the boundary of this disk and these arcs along the boundary of a disk are pairwise disjoint.
						
			Denote the neighbors of some vertex $q$ by $q_0,\ldots,q_{l-1}$ in clockwise order around $q$.
			We choose $l-1$ points, called \emph{ports}, on the small arc which is the intersection of the inflated edge $\{q_i,q\}$ and $C(q)$ for each $i$ (so in total $l(l-1)$ ports).
			We denote the ports by $q_{i,j}$ ($j=0,\ldots,i-1,i+1,\ldots l-1$) and order them counter-clockwise around $q$, such that $q_{i,i+1}$ is the first port, $q_{i,i+2}$ is the second port and so on (addition is modulo $l$).
			Every edge $\{p_{\{q_i,q\}},p_{\{q,q_j\}}\}$ of $J$ 
			is drawn as follows: $p_{\{q_i,q\}}$ is connected to $q_{i,j}$ with a curve $\gamma_{i,j}$ inside the inflated $\{q_i,q\}$, then $q_{i,j}$ is connected to $q_{j,i}$ inside $C(q)$ with a straight-line segment, and finally $q_{j,i}$ is connected to $p_{\{q_j,q\}}$ with a curve $\gamma_{j,i}$ inside the inflated $\{q_j,q\}$.
			Due to the ordering of the ports, it is easy to achieve that for every $i$ the curves $\gamma_{i,j}$ do not cross inside the inflated edge $\{q_i,q\}$.
			It is also easy to obtain that these curves do not cross the curves that go from $p_{q_i,q}$ towards $q_i$ (which correspond to the edges of the line graph between the edges incident to $q_i$, that is, edges for which $q_i$ plays the role of $q$ in the current description).
			Also, due to the ordering of the ports, drawings of $q_iq_j$ and $q_iq_k$ (for any $i,j,k$) do not intersect inside $C(q)$.

			We also need to show that no segments cross inside $C(q)$.
			Suppose that this is not the case and denote the endpoints of the respective edges (after a suitable reindexing) by $q_1,q_2,q_3,q_4$.
			Let $D$ and $D'$ be the pseudo-disks in $\F$ whose intersections with $S$ are $\{q_1,q,q_3\}$ and $\{q_2,q,q_4\}$, respectively.
			Since the Delaunay-edges are drawn such that they are contained in every pseudo-disk that contains their endpoints, and by our choice $C(q)$ lies inside every pseudo-disk that contains $q$, it follows that our drawn curves lie inside $D$ and $D'$, respectively.
			Moreover, by Lemma~\ref{lem:buzaglo} they must cross an even number of times.
			Since the parts outside $C(q)$ are disjoint, and inside $C(q)$ the two segments can cross at most once, we get a contradiction.
			Therefore, we have proved that with this drawing $J$ is a plane graph.\qedwhite
		\end{proof}
		
	Since $J$ is a planar graph, it is $4$-colorable. We use a $4$-coloring of $J$ to color the Delaunay-edges of $D$.
	Now suppose that $D \in \F$ is a pseudo-disk that contains at least three points.
	Since $\F$ is shrinkable, it follows that there is a pseudo-disk $D' \in \F$ that contains exactly three points from $S \cap \F$.
	Lemma~\ref{lem:shrink} also implies that each of the three points from $S$ inside $D'$ is connected by a Delaunay-edge to one of the other two points. Therefore, $D'$ contains two Delaunay-edges that share a common endpoint and thus are colored with different colors.\qedwhite
\end{proof}

Note that Theorem~\ref{thm:pseudo} implies the same for saturated pseudo-disk families as a saturated family is always shrinkable. The following result concerning homothets of a convex compact set is also implied by Theorem~\ref{thm:pseudo}.

Let $C$ be a convex compact set and let $\partial C$ denote its boundary. We say that a set of points $S$ is in \emph{general position} with respect to $C$, if there is no homothet of $C$ that contains more than three points of $S$ on its boundary and no two points in $S$ define a line that is parallel to a line-segment which is contained in $\partial C$.

\begin{corollary} \label{cor:homothets}
	Let $C$ be a convex compact set and let $S$ be a finite planar set of points that is in general position with respect to $C$.
	Then it is possible to $4$-color the Delaunay-edges of $S$ with respect to the family of all homothets of $C$ such that every homothet of $C$ that contains at least three points from $S$ contains Delaunay-edges of different colors.
\end{corollary}

\begin{proof}
	Let $\F$ be a finite subfamily of the set of homothets of $C$ such that if $C'$ is a homothet of $C$, then there is $C'' \in \F$ such that $C'' \cap S = C' \cap S$. One can show along the lines of the proof of Lemma~2.3 in~\cite{AKV15} that this subfamily is shrinkable with respect to $S$. 
	
	We may assume that for every $C_1,C_2 \in \F$ the boundaries
	$\partial C_1$ and $\partial C_2$ intersect a finite number of times.
	Indeed, if there are overlaps between boundaries, then we can slightly inflate one of the homothets, say $C_1$, without changing $C_1 \cap S$ since $C_1$ is closed.
	
	It follows from the next folklore lemma (see, e.g., \cite[Corollary 2.1.2.2]{Ma2000}), that $\F$ is a family of pseudo-disks.
	
	\begin{lemma}
		\label{lem:pseudo-disks}
		Let $C$ be a convex and compact set and let $C_1$ and $C_2$ be homothets of $C$.
		Then if $\partial C_1$ and $\partial C_2$ intersect finitely many times, then they intersect in at most two points.
	\end{lemma}
	
	Since $\F$ is also shrinkable, the corollary now follows from Theorem~\ref{thm:pseudo}.\qedwhite
\end{proof}

\subsection{Coloring Delaunay-edges with respect to halfplanes}\label{sec:hp}

Let $S$ be a finite set of points in the plane and let $\F$ be the family of all halfplanes. Note that a halfplane may contain an arbitrary number of points from $S$ while containing only one Delaunay-edge of $D(S,\F)$, see Figure~\ref{fig:halfplanes-many} for an example.
\begin{figure}
	\centering
	\begin{subfigure}[b]{0.4\textwidth}
		\includegraphics[width=\textwidth]{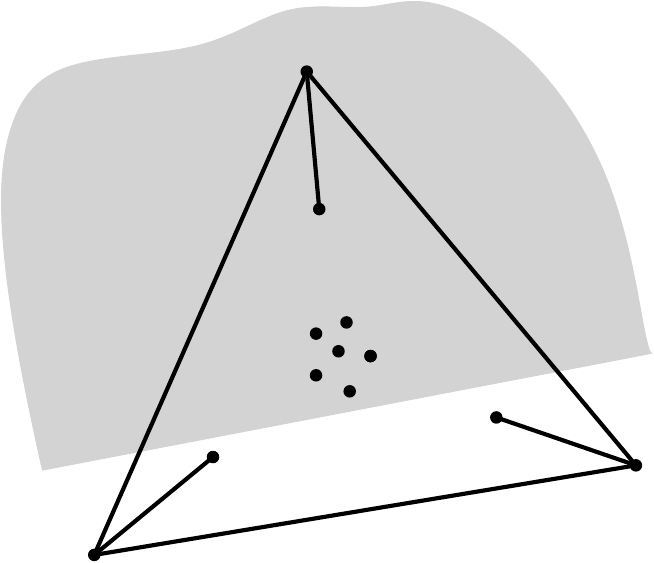}
		\caption{}
		\label{fig:halfplanes-many}
	\end{subfigure}
	\hspace{1cm}
	\begin{subfigure}[b]{0.4\textwidth}
		\includegraphics[width=\textwidth]{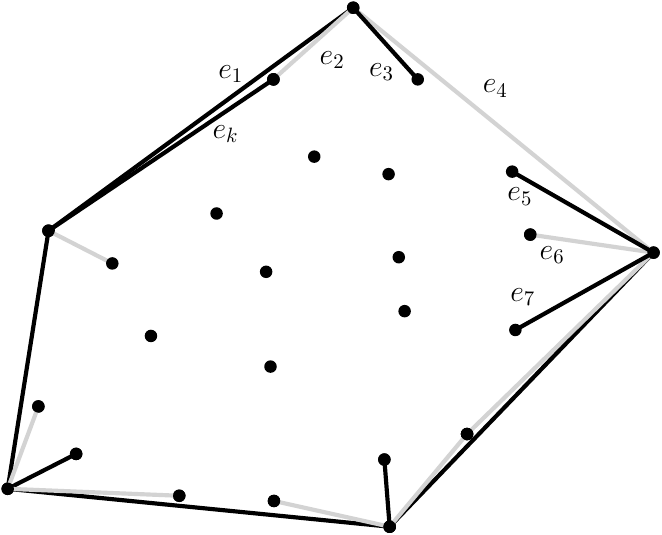}
		\caption{}
		\label{fig:halfplanes}
	\end{subfigure}
	\caption{(a) A halfplane may contain many points but only one Delaunay-edge. (b) Coloring the Delaunay-edges.}
\end{figure}
Therefore a claim along the lines of, e.g., Theorem~\ref{thm:bless} is impossible.
Thus, Theorem~\ref{thm:halfplanes} states that the edges of the Delaunay-graph $D(S,\F)$ can be $2$-colored such that every halfplane that contains at least three \emph{Delaunay-edges} contains two differently colored edges.

\begin{proof}[Proof of Theorem~\ref{thm:halfplanes}]
	For simplicity we assume that the points are in general position, as it is easy to modify the proof so that it works for points that are not in general position.
	
	Let $S$ be a finite set of points in the plane and let $\F$ be the family of all halfplanes. The convex hull of $S$, $\CH(S)$, is a convex polygon whose vertices are points from $S$. Note that for each of the Delaunay-edges in the Delaunay-graph $D(S,\F)$ at least one of the two endpoints is a vertex of $\CH(S)$.
	Thus, $D(S,\F)$ consists of the sides of $\CH(S)$ and stars centered at each of the vertices of $\CH(S)$.
	
	Starting from an arbitrary vertex of $\CH(S)$ we go over all the Delaunay-edges adjacent to this vertex in a counter-clockwise order and color them alternatively with red and blue. Then we move to the next vertex of $\CH(S)$ in a clockwise order.
	Repeat the same process (the first edge adjacent to it is already colored), and so on until all the edges are colored (see Figure~\ref{fig:halfplanes} for an example).
	
	Let $e_1,e_2,\ldots,e_k$ be the Delaunay-edges listed in the order they were visited by the coloring algorithm.
	Note that every pair of consecutive edges are of different colors, except perhaps $e_1$ and $e_k$ which are of the same color if $k$ is odd.
	
	\begin{proposition}\label{prop:consecutive}
		If a halfplane $H$ contains two Delaunay-edges $e_i$ and $e_j$, $i<j$, then $H$ contains the edges $e_{j+1},\ldots,e_k,e_1,\ldots,e_{i-1}$ or the edges $e_{i+1},\ldots,e_{j-1}$.
	\end{proposition}
	
	\begin{proof}
		Suppose first that $e_i$ and $e_j$ share a vertex of $\CH(S)$ as an endpoint and denote it by $v$.
		Assume without loss of generality that all the edges $e_i,\ldots,e_j$ are incident to $v$ and let $v_l$ denote the other endpoint of $e_l$, $l=i,\ldots,j$.
		Assume for contradiction that $H$ does not contain $e_l$ for some  $i < l < j$.
		It follows that $v, v_i, v_l, v_j$ are in convex position such that this is also their clockwise or counter-clockwise order around their convex hull.
		However, then there cannot be a halfplane that separates $\{v,v_l\}$  from $\{v_i,v_j\}$ contradicting the fact that $\{v,v_l\}$ is a Delaunay-edge.
		
		Suppose now that $e_i$ and $e_j$ are disjoint and let $u$ and $u'$ be vertices of $\CH(S)$ such that $u$ is an endpoint of $e_i$ and $u'$ is an endpoint of $e_j$.
		Let $U$ be the vertices of $\CH(S)$ in the clockwise order from $u$ to $u'$, and assume without loss of generality that $H$ contains all of them (this can be assumed since if $H$ does not contain all of them, then it must contain all the vertices of $\CH(S)$ from $u'$ to $u$).
		We may also assume that $e_{i+1},\ldots,e_{j-1}$ are incident to the vertices in $U$.
		
		We claim that $H$ must contain each of these edges.
		Indeed, suppose that $H$ does not contain $e_l$, for some $i < l < j$.
		If none of $u$ and $u'$ is an endpoint of $e_l$, then let $u'' \in U$ be one of its endpoints and let $v_l$ be its other endpoint.
		As in the previous case, it follows that $u,u'',u',v_l$ are in convex position such that this is also their clockwise order around their convex hull.
		However, then there cannot be a halfplane that separates $\{u'',v_l\}$  from $\{u,u'\}$.
		Finally, suppose that one endpoint of $e_l$ is, without loss of generality, $u$ and let $v_l$ be its other endpoint. Denote also by $v_i$ the other endpoint of $e_i$.
		Then, as before, it follows that $u,v_i,v_l,u'$ are in convex position in this cyclic order. However, there should be a halfplane that separates $\{u,v_l\}$ from the other two vertices which is impossible.\qedwhite
	\end{proof}
	
	Suppose that $H \in \F$ is a halfplane that contains at least three Delaunay-edges. Then it follows from Proposition~\ref{prop:consecutive} that $H$ contains at least three consecutive Delaunay-edges and therefore it contains two edges of different colors. 
	
	To see that the constant three in the theorem is tight,
	note that when $S$ is a set of an odd number of points in convex position, then it is impossible to $2$-color the Delaunay-edges of $D(S,\F)$ such that every halfplane that contains two Delaunay-edges contains edges of both colors. \qedwhite
\end{proof}

\subsection{Coloring Delaunay-edges with respect to bottomless rectangles}\label{sec:bless}

A \emph{bottomless} rectangle is the set of points $\{(x,y) \mid a \le x \le b, y \le c\}$ for some numbers $a,b,c$ such that $a \le b$.
Let $\F$ be the set of bottomless rectangles and let $S$ be a finite set of points in the plane.
As in the previous section, we may assume that no two points in $S$ share the same $x$- or $y$-coordinates.
It is known~\cite{K11} that it is possible to color the points in $S$ with two colors such that every bottomless rectangle that contains at least four points from $S$ contains two points of different colors.
This easily implies that the Delaunay-edges of $D(S,\F)$ can be colored with two colors such that every bottomless rectangle that contains at least five points from $S$ contains two Delaunay-edges of different colors.
Indeed, color first the points of $S$ such that every bottomless rectangle that contains four points from $S$ is non-monochromatic.
Then color every Delaunay-edge with the color of its left endpoint.
Suppose that $R \in \F$ contains at least five points from $S$ and let $p_i$ be the $i$th point in $R \cap S$ from left to right.
There is a bottomless rectangle $R'$ such that $R' \cap S = \{p_1,p_2,p_3,p_4\}$, therefore there are two differently colored points among these points.
Since $\{p_i,p_{i+1}\}$ is a Delaunay-edge for every $i$, it follows that $R$ contains two Delaunay-edges of different colors.

Next we prove Theorem~\ref{thm:bless}. Namely, that containing four points from $S$ suffices for a bottomless rectangle to contain Delaunay-edges of different colors and this is tight. The proof is similar to proof of the aforementioned statement in~\cite{K11}.

\begin{proof}[Proof of Theorem~\ref{thm:bless}]
	We add the points in $S$ one-by-one from bottom to top.
	Note that any Delaunay-edge that is introduced when a point is added,
	remains a Delaunay-edge when adding the remaining points.
	
	At any moment call an edge \emph{neighborly} if it connects points that are next to each other in the horizontal order of points.
	Denote the points added so far by their order from left to right, $p_1,p_2,\ldots,p_k$, and let $p_i$ be the last point to be added.
	Thus, the neighborly edges are $\{p_j,p_{j+1}\}$, for $j=1,\ldots,k-1$.
	We maintain the invariant that there are no three consecutive neighborly edges of the same color.
	This is trivial to obtain for $k \le 3$, so assume that $k \ge 4$.
	When $p_i$ is added, it introduces either one (if $i=1$ or $i=k$) or two new neighborly edges (and causes an existing edge to cease being a neighborly edge).
	In the first case we color the new edge with a color that is different from the color of the single neighborly edge that shares an endpoint with the new edge.
	In the second case, if $i=2$ (resp., $i=k-1$), then we color the new edges $\{p_1,p_2\}$ and $\{p_2,p_3\}$ (resp., $\{p_{k-2},p_{k-1}\}$ and $\{p_{k-1},p_{k}\}$)
	with a color that is different from the color of $\{p_{3},p_{4}\}$ (resp., $\{p_{k-3},p_{k-2}\}$).
	When $2 < i < k-1$ there are two cases to consider:
	If $\{p_{i-2},p_{i-1}\}$ and $\{p_{i+1},p_{i+2}\}$ are of the same color, then we color the new edges $\{p_{i-1},p_{i}\}$ and $\{p_{i},p_{i+1}\}$ with the opposite color. 
	Otherwise, if $\{p_{i-2},p_{i-1}\}$ and $\{p_{i+1},p_{i+2}\}$ are of different colors, then $\{p_{i-1},p_{i}\}$ is colored with the color of $\{p_{i+1},p_{i+2}\}$ and $\{p_{i},p_{i+1}\}$ is colored with the color of  $\{p_{i-2},p_{i-1}\}$.
	It is easy to verify that the invariant that there are no three consecutive neighborly edges of the same color is maintained in all cases. 
	
	Suppose now that $R$ is a bottomless rectangle that contains at least four points.
	When the topmost point from $S \cap R$ was added, then all the points in $S \cap R$ formed $|S \cap R|-1 \ge 3$ consecutive neighborly edges.
	Therefore, not all of these edges are of the same color.
	Since these edges remain Delaunay-edges (perhaps non-neighborly) when the remaining points are added, $R$ contains Delaunay-edges of both colors.
	
	To see that it is not enough to consider bottomless rectangles that contain three points from $S$, consider the point set in Figure~\ref{fig:bless3} and assume it is colored as required.
	\begin{figure}
		\centering
		\includegraphics[width=5cm]{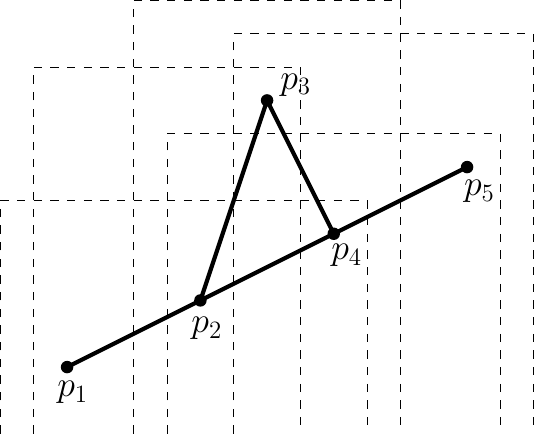}
		\caption{Considering bottomless rectangles containing three points is not enough.}
		\label{fig:bless3}
	\end{figure}
	Assume without loss of generality that $\{p_1,p_2\}$ is red.
	Then $\{p_2,p_3\}$ and $\{p_2,p_4\}$ must be both blue.
	This implies that $\{p_4,p_5\}$ must be red, and so $\{p_3,p_4\}$ must be blue.
	Therefore, all the Delaunay-edges in the bottomless rectangle that contains $p_2$, $p_3$ and $p_4$ are blue, a contradiction. \qedwhite
\end{proof}

We remark that it is possible to modify the proof and show that one can two-color the Delaunay-edges such that any bottomless rectangle that contains at least three edges contains edges of both colors.
Specifically, one can define a subgraph $J$ of the line graph of the Delaunay-graph, similarly to the proof of Theorem~\ref{thm:pseudo}, and show that this graph is `doubly $2$-degenerate', by which we mean that we can recursively find two adjacent degree-$2$ vertices in it. Then, it follows by induction that we can two-color the vertices of $J$ such that any connected component on at least three vertices is non-monochromatic.

\subsection{Coloring Delaunay-edges with respect to axis-parallel rectangles}\label{sec:ap}

For a partially ordered set $P=(X,\le)$ we say that $x \in X$ is an \emph{immediate predecessor} of $y \in X$ if $x < y$ and there is no $z \in X$ such that $x < z < y$ (where $x < y$ means that $x \le y$ and $x \ne y$).
The \emph{directed Hasse diagram} of $P$ is the directed graph $(X, \{ (x,y) \mid x$ is an immediate predecessor of $y\})$.
The following is easy to prove.

\begin{lemma}\label{lem:hasse}
	Let $P=(X,\le)$ be a partially ordered set such that $X$ is finite.
	Then it is possible to color the edges of the directed Hasse diagram of $P$
	with $\lceil \log_2 |X| \rceil$ colors such that there is no monochromatic directed path of length two.
\end{lemma}

\begin{proof}
	Denote $n=|X|$.
	It is enough to prove the lemma for $n=2^k$, $k \in \mathbb{N}$.
	We prove by induction on $k$.
	For $k=0$ the lemma holds.
	Let $P=(X,\le)$ be a poset where $|X|=n=2^k$ and $k>0$.
	Let $P^*=(X,\le^*)$ be a linear extension of $P$ and let $x_1 \le^* x_2 \le^* \ldots \le^* x_n$ be the elements of $X$ ordered according to $P^*$.
	Set $P_1=(X_1=\{x_1,\ldots,x_{n/2}\},\le)$ and  $P_2=(X_2=\{x_{n/2+1},\ldots,x_{n}\},\le)$.
	Note that the directed Hasse diagram of $P$ consists of the directed Hasse diagrams of $P_1$ and $P_2$ and edges of the form $(x,y)$ such that $x_i \in X_1$ and $y \in X_2$.
	By the induction hypothesis the edge set of each of the directed Hasse diagrams of $P_1$ and $P_2$ can be colored with (the same set of) $k-1$ colors such that there is no monochromatic path of length two.
	
	By using the same coloring for the edges of the directed Hasse diagram of $P$ and introducing a new color for the edges $\{(x,y) \mid x \in X_1, y \in X_2\}$ we obtain the required coloring. \qedwhite
\end{proof}

\paragraph{Remark.}
The number of colors in Lemma~\ref{lem:hasse} cannot be an absolute constant.
In fact, for any integer $d>0$ there is no integer $c=c(d)$ such that it is possible
to color the edges of the directed Hasse diagram of any finite poset with $c$ colors
such that there is no monochromatic path of length $d$.
To see this, observe that if such an integer would exist, then it would be possible to properly color the vertices of the Hasse diagram of any poset $P$ with $O(1)$ colors which is known to be false~\cite{chenpach,KN91}.
Indeed, it is easy to verify that by assigning each vertex $v$ the color $(c',d')$ where $d'$ is the length of the longest monochromatic path ending at $v$ (in the directed Hasse diagram of $P$) and $c'$ is the color of that path one obtains a proper coloring with $O(1)$ colors.

\begin{corollary}
	Let $S$ be a finite set of $n$ points in the plane and let $\F$ be the family of
	axis-parallel rectangles. 
	Then it is possible to color the edges of the (non-planar) Delaunay-graph $D(S,\F)$ with $O(\log n)$ colors such that every axis-parallel rectangle that contains at least three points from $S$ contains Delaunay-edges of different colors.
\end{corollary}

\begin{proof}
	We may assume without loss of generality that no two points in $S$ share the same $x$- or $y$-coordinates.\footnote{Note that we cannot just take an arbitrary perturbation, as we did for points. For example, if we have 3 collinear points inside a rectangle (thus two Delaunay-edges) we might introduce a new Delaunay-edge with an arbitrary perturbation.} Indeed, this can be obtained by a mapping $(x,y) \mapsto (x+\varepsilon y, y+\varepsilon x)$ for a very small $\varepsilon$ (that depends on $S$) without changing the hyperedges of $G(S,\F)$.
	
	Orient every edge of $D(S,\F)$ from left to right.
	Let $P_1=(S,\prec_1)$ where $(x,y) \prec_1 (x',y')$ if $x < x'$ and $y < y'$.
	Let $P_2=(S,\prec_2)$ where $(x,y) \prec_2 (x',y')$ if $x < x'$ and $y > y'$.
	Note that the edge set of the oriented Delaunay-graph is the union of the (disjoint) edge sets of the directed Hasse diagrams of $P_1$ and $P_2$.
	Therefore, it follows from Lemma~\ref{lem:hasse} that we can color these edges with $2\lceil\log_2 n\rceil$ colors such that there is no monochromatic path of length two.
	
	Let $R' \in \F$ be an axis-parallel rectangle that contains at least three points from $S$ and let $p_1,p_2,p_3$ be the leftmost points in $R'$, ordered from left to right.
	Then $(p_1,p_2)$ and $(p_2,p_3)$ are edges in the oriented Delaunay-graph $D(S,\F)$.
	If they belong to different Hasse diagrams, then they are of different colors.
	Otherwise, if they belong to the same Hasse diagram, then they are of different colors since they form a path of length two. \qedwhite
\end{proof}

\section{Coloring all \texorpdfstring{$t$}{t}-subsets of abstract hypergraphs}\label{sec:allthyp}

Henceforth we consider colorings of \emph{all} the $t$-subsets of points.
In this section we show some general results that do not rely on geometric assumptions. 
First, intuitively, as the number of all $t$-subsets increases considerably by increasing $t$, the coloring problem should become easier or at least not become (significantly) harder.
The next claim confirms this intuition in the following sense: proper $k$-colorability of $t$-subsets implies proper $2$-colorability of $t'$-subsets for $t'>t$ (while the corresponding $m$ might increase).

\begin{proposition}\label{prop:chi}
	For every fixed positive integers $m$, $k$, $t$ and $t'>t$ there exists an integer $m'=m'(m,k,t,t')$ with the following property. Let $G$ be a hypergraph such that the $t$-subsets of the vertices of $G$ can be $k$-colored such that every hyperedge that contains at least $m$ vertices contains two differently colored $t$-subsets.
	Then the $t'$-subsets of the vertices of $G$ can be $2$-colored such that every hyperedge that contains at least $m'$ vertices contains two differently colored $t'$-subsets.
\end{proposition}

\begin{proof}[Proof of Proposition~\ref{prop:chi}]
	We construct a coloring $c'$ of $t'$-subsets from the coloring $c$ of $t$-subsets.
	The color of a $t'$-subset $T$ is red if every $t$-subset of the vertices of $T$ is colored with the same color by $c$, and blue otherwise.
	By Ramsey's Theorem (for hypergraphs) we can choose $m'\ge m$ such that in any coloring of all $t$-subsets of a set of size $m'$ there is a monochromatic $t'$-subset, i.e., whose every subset of size $t$ got the same color.
	This way in $c'$ every set of size at least $m'$ (and so also every hyperedge of size at least $m'$) contains a red $t'$-subset. Also, in the coloring $c$ every hyperedge in $G$ of size $m\le m'$ was not monochromatic, thus such a hyperedge also contains a blue $t'$-subset.\qedwhite
\end{proof}

One can of course get better bounds on $m'$ if we construct a coloring more carefully instead of using Ramsey's theorem.
The following is such a claim for $t=1$ about polychromatic $k$-colorability.

\begin{proposition}\label{prop:mk}
	For every fixed positive integers $m, k$ and $t'>1$ there exists an integer $m'=m'(m,k,t')$ with the following property.
	Let $G=(V,E)$ be a hypergraph such that the vertices of $G$ can be $k$-colored such that every hyperedge of $G$ that contains at least $m$ vertices contains a vertex of each of the $k$ colors.
	Then for any $k'\le \sum_{i=0}^{t'} \binom{k-1}i$	
	there exists a $k'$-coloring of the $t'$-subsets of the vertices of $G$ such that every hyperedge of $G$ that contains at least $m'$ vertices contains a $t'$-subset of each of the $k'$ colors.
	Moreover, this is best possible, i.e., the statement does not hold for $k'\ge \sum_{i=0}^{t'} \binom{k-1}i+1$.
	Also, we can pick $m'=\max\{m,k(t'-1)+1\}$.
\end{proposition}

Note that $\sum_{i=0}^{t'} \binom{k-1}i$ is always more than $k$, but never exceeds $2^{k-1}$. 
We will denote by $[n]$ the set $\{1,2,\ldots,n\}$ and by ${A}\choose{k}$ the $k$-element subsets
of a set $A$.

\begin{proof}
	Note that by definition of polychromatic colorings, it is enough to prove Proposition \ref{prop:mk} for $k'=\sum_{i=0}^{t'} \binom{k-1}i$.	
	Let $c:V \rightarrow [k]$ be a $k$-coloring of the vertices of $G$ as above, 
	and set $k'=\sum_{i=0}^{t'} \binom{k-1}i=\binom k{t'}+\sum_{i=0}^{t'-2} \binom{k-1}i$.
	We will construct a $k'$-coloring $c'$ of the $t'$-subsets of vertices of $G$ such that $c'$ is a mapping to the union of the $t'$-element subsets of $[k]$ and the $i$-element subsets of $[k-1]$, for $i=0,\ldots,t'-2$. Thus, there are indeed $k'$ colors.
	
	The color of each $t'$-subset $T$ will be determined by the multiset of the colors of its vertices, $c(T):=\{c(v) \mid v\in T\}$, as follows:
	
	\begin{itemize}
		\item If no color appears twice in $c(T)$ (only possible if $t'\le k$), then $c'(T)$ is the $t'$-element subset of $[k]$ defined by those colors, i.e., $c'(T)=c(T)$.
		\item If exactly one color (say, $r$) appears at least twice in $c(T)$ and each of the other $i$ colors in $c(T)$ appears once, then $c'(T)$ will be an $i$-element subset of $[k-1]$, determined as follows. Define a bijection $\varphi_r$ between $[k]\setminus \{r\}$ and $[k-1]$ such that $\varphi_r(j)=j-r$ for $r+1\le j\le k$ and $\varphi_r(j)=j-r+k$ for $1\le j\le r-1$. Then $c'(T)=\{\varphi_r(c(v)) \mid v\in T, c(v) \ne r\}$.
		\item Otherwise, define $c'(T)$ arbitrarily.
	\end{itemize}
	
	Now we need to show that all $k'$ colors appear in a hyperedge $e$ of size $m'=\max\{m,k(t'-1)+1\}$.
	Note that among the vertices of $e$ we have all $k$ colors from $[k]$ since $m' \ge m$.
	Let $U=\{u_1,u_2,\ldots,u_k\} \subseteq e$ be a set of vertices such that $c(v_i)=i$.
	If $t' \le k$, then the $\binom{k}{t'}$ $t'$-subsets of vertices from $U$ are colored by all colors in  $\binom{[k]}{t'}$.
	
	Now let $A$ be an $i$-element subset of $[k-1]$, for some $0 \le i \le t'-2$.
	By the pigeonhole principle there is a subset $U' \subseteq e$ of $t'$ vertices that are colored by the same color $r \in [k]$.
	Suppose that $T$ is a $t'$-subset that consists of $t'-i \geq 2$ vertices from $U'$
	and every vertex $u_j \in U$ such that $\varphi_r(j) \in A$.
	Since $\varphi_r$ is a bijection there are $i$ such vertices and it follows that $c'(T)=A$.
	
	To see that for $k'=\sum_{i=0}^{t'} \binom{k-1}i+1$ there might not be such a coloring for any $m'$, consider the following hypergraph $G=(V,E)$.
	$V$ consists of two parts, $U=\{u_1,u_2,\ldots,u_{k-1}\}$ and $W=\{w_1,w_2,\ldots,w_{N}\}$, where $N$ is some sufficiently large number.
	The hyperedges $E$ are all supersets of $U$, i.e., they all contain every vertex from $U$, and an arbitrary subset of $W$.
	The coloring $c(u_i)=i, c(w_j)=k$ is a polychromatic $k$-coloring of $G$ for $m=k$.
	Suppose for a contradiction that there is an integer $m'$ for which there exists a $k'$-coloring $c'$ of the $t'$-subsets of the vertices of $G$ such that every hyperedge of $G$ that contains at least $m'$ vertices contains a $t'$-subset of each of the $k'$ colors.
	For each $X\subset U$, define a coloring $c_X$ of the $(t'-|X|)$-subsets of $W$ so that for a subset $Y$ of size $t'-|X|$ we have $c_X(Y)=c'(X\cup Y)$.
	
	Let $X_1,X_2,\ldots,X_l$ be the subsets of $U$ of size at most $t'-1$, $l = \sum_{i=0}^{t'-1} \binom{k-1}i$.
	For any integer $N_1$ we can choose $N$ to be large enough such that it 
	follows from the hypergraph Ramsey Theorem that there is a subset $W_1 \subset W$
	of size $N_1$ such that all of its subsets of size $t'-|X_1|$ are of the same color with respect to $c_{X_1}$.
	Thus we may say that this color (of $c'$) is associated with $X_1$.
	Similarly, for any integer $N_2$ we can choose $N$ to be large enough such that there is a subset $W_2 \subset W_1$ such that all of its subsets of size $t'-|X_2|$ are of the same color with respect to $c_{X_2}$.
	Thus we may say that this color (of $c'$) is associated with $X_2$.
	Continuing in this manner, we may choose $N$ to be large enough and obtain a subset
	$W' = W_l\subset W_{l-1}$ of size $m'-k+1$ such that every $t'$-subset of $U \cup W'$ has a color (according to $c'$) that is associated with some $i$-subset of $U$, where $i \le t'$.
	However, this implies that the $m'$-subset $U \cup W'$ has $t'$-subsets of at most $\sum_{i=0}^{t'} \binom{k-1}i < k'$ different colors.
	\qedwhite
\end{proof}

Surprisingly, obtaining a polychromatic $k$-coloring in a similar way when $t>1$ fails already for $t=2$, i.e., when the condition is that pairs of vertices have a suitable coloring. Note that it follows from Proposition \ref{prop:mk} that if for $t=1$ and $k=3$ we have a polychromatic $k$-coloring of the $t$-subsets (i.e., the vertices), then for $t'=3$ and $k'=4$ (and thus also for every $k'\le 4$) we have a polychromatic $k'$-coloring of the $t'$-subsets. Below we give an example showing that such an implication does not hold for $t=2$, $k=3$ and $t'=3$, $k'=3$. 

Suppose we are given a $3$-coloring of every pair of the vertex set of a hypergraph $G$ such that every hyperedge that contains at least $m$ vertices contains a pair of each of the $3$ colors, and we want to construct a $3$-coloring of the triples of vertices with the same property (possibly for a larger $m'$).
We show that this is not possible if the color of each triple depends only on the colors of the three pairs in it:

\begin{proposition}\label{prop:nocol}
	There is no mapping $\hat c$ from the $3$-element multisets of $[3]$ to $[3]$ such that the following holds for some $m'(m)$.
	For every hypergraph $G=(V,E)$ if $c:{{V}\choose{2}} \rightarrow [3]$ is a $3$-coloring of the pairs of vertices of $G$ such that every hyperedge that contains at least $m$ vertices contains pairs of all three colors, then $c$ and $\hat c$ defines a $3$-coloring $c':{{V}\choose{3}} \rightarrow [3]$ of the triples of vertices of $G$ such that every hyperedge that contains at least $m'(m)$ vertices contains triples of all three colors (where $c'(\{x,y,z\})=\hat c(\{c(\{x,y\}),c(\{x,z\}),c(\{y,z\})\})$ for every triple of vertices $\{x,y,z\}$). 
\end{proposition}

\begin{proof}
	Suppose for contradiction that $\hat c$ and $m'=m'(m)$ as above exist.
	Note that there are ten $3$-element multisets of $[3]$, that is, 
	there are ten possible colorings of the three pairs of a triple by $3$-colors: $\{1,1,1\}$, $\{1,1,2\}$, $\{1,1,3\}$, $\{2,2,1\}$, $\{2,2,2\}$, $\{2,2,3\}$, $\{3,3,1\}$, $\{3,3,2\}$,  $\{3,3,3\}$ and $\{1,2,3\}$.
	
	Let $G_1=(V_1,E_1)$ be a hypergraph such that $|V_1| = \max\{m,m'\}$ and let $M_1 \subseteq {{V_1}\choose{2}}$ and $M_2 \subseteq {{V_1}\choose{2}}$ be two non-empty sets of pairs of vertices such that no vertex belongs to two pairs in $M_1 \cup M_2$, i.e., $M_1 \cup M_2$ is a matching.
	Let $E_1$ consist of the hyperedge $V_1$ (all the vertices) and all the subsets of $V_1$ of size at least $\max\{m,4\}$ that contain a pair from $M_1$ and a pair from $M_2$.
	
	Suppose that $c_1$ is a $3$-coloring of the pairs in ${{V_1}\choose{2}}$ such that the color of every pair in $M_i$ is $i$, for $i=1,2$, and the color of every other pair is $3$.
	Then clearly every hyperedge of $G_1$ of size at least $m$ contains a pair of each of the colors.
	Note that the colors of the three pairs defined by a triple of vertices are either
	$\{3,3,1\}$, $\{3,3,2\}$ or $\{3,3,3\}$.
	Consider the coloring of the triples as defined by $c_1$ and $\hat c$.
	Since the hyperedge $V_1$ must contain a triple of every color, it follows that $\hat c$ maps 
	each of the above $3$-multisets of colors to a different color.
	
	Suppose now that $G_2=(V_2,E_2)$ is a hypergraph such that $|V_2| = 2\cdot\max\{m,m'(m)\}$ and let $A=\{a_1,a_2,\ldots,a_{|V_2|/2}\}$ and $B=\{b_1,b_2,\ldots,b_{|V_2|/2}\}$ be two disjoint subsets of vertices.
	Let $E_2$ consist of the hyperedge $V_2$ (all the vertices) and all the hyperedges $A \cup \{b_i\}$ and $B \cup \{a_i\}$, for some $i$.
	Let $P_1$ be the pairs of vertices of the form $\{a_i,a_j\}$ or $\{b_i,b_j\}$ (for $i \ne j$),
	Let $P_2$ be the pairs of vertices of the form $\{a_i,b_i\}$,
	and let $P_3$ be the (remaining) pairs of vertices of the form $\{a_i,b_j\}$ (for $i \ne j$).
	
	Suppose that $c_2$ is a $3$-coloring of the pairs in ${{V_2}\choose{2}}$ such that the color 
	of every pair in $P_i$ is $i$, for $i=1,2,3$.
	Then clearly every hyperedge of $G_2$ of size at least $m$ contains a pair of each of the colors.
	Note that the colors of the three pairs defined by a triple of vertices are either
	$\{1,1,1\}$, $\{1,2,3\}$ or $\{3,3,1\}$.
	Consider the coloring of the triples as defined by $c_2$ and $\hat c$.
	Since the hyperedge $V_2$ must contain a triple of every color, it follows that $\hat c$ maps 
	each of the above $3$-multisets of colors to a different color.
	Therefore, 	$\{1,2,3\}$ is mapped to a color different from the colors of $\{1,1,1\}$ and $\{3,3,1\}$.
	Similarly, if we color the pairs in $P_1$ with $2$, the pairs in $P_2$ with $1$ and the pairs in $P_3$ with $3$, then we conclude that $\hat c$ maps $\{1,2,3\}$ to a color different from the colors of $\{2,2,2\}$ and $\{3,3,2\}$.
	And if we color the pairs in $P_1$ with $3$, the pairs in $P_2$ with $2$ and the pairs in $P_3$ with $1$, then we conclude that $\hat c$ maps $\{1,2,3\}$ to a color different from the colors of $\{3,3,3\}$ and $\{1,1,3\}$.
	Thus, according to $\hat c$ the color of $\{1,2,3\}$ is different from the colors of $\{3,3,1\}$, $\{3,3,2\}$ and $\{3,3,3\}$. However, we observed before that $\hat c$ maps these three $3$-multisets of colors to different colors, a contradiction.\qedwhite
\end{proof}

It is possible to change the constructions in the proof of Proposition~\ref{prop:nocol} such that every hyperedge that contains at least $m$ vertices contains many pairs of each of the colors, so even such an assumption is not strong enough to guarantee a coloring for triples. 
One might try to get a $k$-coloring of $t'$-subsets from stronger conditions on the coloring of the $t$-subsets, however, we do not further study this question.

\section{Coloring all $t$-subsets of geometric hypergraphs}\label{sec:alltgeom}

In this section we consider mainly three (closely related) types of families of geometric regions, $\HH$-regions, homothets of a convex polytope and axis-parallel boxes.

Given $\cal F$ and $S$, we want to color all $t$-subsets of $S$ such that no member of $\cal F$ that has many points of $S$ is monochromatic, i.e., it contains two $t$-subsets that are colored with different colors.

\begin{definition}
	Given a finite family of halfspaces $\HH=\{H_1,\dots, H_h\}$ we call a region $R$ an \emph{$\HH$-region} if it is the intersection of finitely many halfspaces, such that each of them is a translate of one of the halfspaces in $\HH$.
\end{definition}
Note that every such region $R$ can be expressed as an intersection of at most $h$ halfplanes (at most one translate of each $H_i$).

Throughout this section we assume, without loss of generality, that any given point set $S$ is in general position in the sense that there are no two points in $S$ that are contained in $\partial H_i$ for some $H_i \in \HH$, where $\partial H_i$ denotes the hyperplane that bounds the halfspace $H_i$.
Indeed, otherwise after a small perturbation of the points of $S$
for every subset $S' \subseteq S$ that is the intersection of $S$ and some $\HH$-region, the perturbed subset $S'$ is still the intersection of (the perturbed) $S$ and some (possibly different) $\HH$-region.

As mentioned in the Introduction, the requirement that $t \geq 2$ in Theorem~\ref{thm:region} is essential.
Indeed, Chen et al.~\cite{chenpach} proved that there is no absolute constant $m$ with the property that every set of points $S$ in the plane can be $2$-colored such that every axis-parallel rectangle that contains at least $m$ points from $S$ contains points of both colors.
Note that axis-parallel rectangles can be defined as the $\HH$-regions of four halfplanes.
Several other constructions about polychromatic colorings/cover-decomposition can also be realized as the $\HH$-regions of four halfplanes; see \cite{PTT05,P10}.
If there are only three halfspaces in $\HH\subset \mathbb{R}^3$ that are in general position, then after an affine transformation the problem becomes equivalent to coloring points with respect to octants (non-general position and different dimension can only make the problem easier).
For this case, it was shown in \cite{KP11} that $m(k,1,3)$ exists for every $k$, i.e., for $|\HH|=3$ any finite collection of points $S$ has a $k$-coloring such that if an $\HH$-region contains at least $m(k,1,3)$ points from $S$, then it contains all $k$ colors. (This implies the same for  $m(k,1,2)$ and  $m(k,1,1)$ although these can be shown directly as well.)

Note that a homothet of a convex polytope $P$ is a $\HH$-region when $\HH$ is the family of all the supporting halfspaces of $P$.
Thus Theorem \ref{thm:region} implies:

\begin{theorem}\label{thm:homoth}
	For every dimension $d$,
	integers $t \geq 2$, $k \ge 1$ and a convex polytope $P$ in $\mathbb{R}^d$ there exists an integer $m=m(k,t,P)$ with the following property: the $t$-subsets of every finite point set $S$ in $\mathbb{R}^d$ can be $k$-colored such that every homothet of $P$ that contains at least $m$ points from $S$ contains a $t$-subset of points of each of the $k$ colors.
\end{theorem}

As mentioned above, it is impossible to $2$-color points (that is, the case $t=1$) with respect to $\HH$-regions
in the plane such that every region that contains a constant number of points (for some absolute constant)
contains a point of each color.
However, for homothets of a convex polygon it might be possible.
The following is an equivalent restatement (by the self-coverability of convex polygons \cite{KP13}) of a conjecture that first appeared in \cite{AKV15}.

\begin{conjecture}\label{conj:poly}
	For every convex polygon $C$ and a positive integer $k$ there exists an integer $m=m(k,C)$ with the following property: the points of every finite point set $S$ in the plane can be $k$-colored such that every homothet of $C$ that contains at least $m$ points from $S$ contains a point of each of the $k$ colors.
\end{conjecture}

We note that Theorem \ref{thm:homoth} for $d=2$ would be implied by Conjecture \ref{conj:poly} using Theorem~\ref{thm:t-ind} (see the proof below), however, Conjecture~\ref{conj:poly} was verified only for triangles~\cite{K13} and parallelograms~\cite{AKV15}.
Recently, a relaxation of this conjecture was proved in \cite{KP16}, namely, that it is possible to $3$-color the points such that each homothet with enough points contains points of at least two colors.
The conjecture fails in $d\ge 3$ for every polytope \cite{P10}.

We have seen that for general hypergraphs we could not show that polychromatic colorability of $t$-subsets implies the same for $t'$-subsets, $t'>t$. On the other hand we can prove Theorem \ref{thm:t-ind} which states this for $\HH$-regions. 

\begin{proof}[Proof of Theorem \ref{thm:t-ind}]
	Let $H$ be a halfspace in $\HH$ and let $H'$ be a translate of $H$  that contains all the points in $S$. 
	We fix an order of the points in $S$ in decreasing order of their distance from $\partial H'$ (recall that the points are in general position and therefore no two points can be at the same distance).
	Consider a coloring of the $t$-subsets having the property assumed by the theorem and  
	let $T = \{p_1,p_2,\ldots,p_t,\ldots,p_{t'}\}$ be a $t'$-subset of points, where the points are listed according to the determined order
	(that is, $p_1$ is the farthest from $\partial H'$ and so on).
	Set the color of $T$ to be the color of the $t$-subset  $\{p_1,p_2,\ldots,p_t\}$.
	
	Let $m=m(k,t,\HH)$ and $m'=m+t'-t$ and suppose that $R$ is an $\HH$-region that contains $l\ge m'$ points from $S$, denote them by $p_1, p_2, \ldots, p_l$ according to their order.
	Since $S$ is in general position, it is possible to continuously translate $H'$ such that it is still a translate of $H$, and obtain a translate $H''$ of $H$ such that $H'' \cap (R \cap S) = \{p_1,p_2,\ldots,p_{l-(t'-t)}\}$.
	Note that $H'' \cap R$ is an $\HH$-region that contains at least $m$ points and therefore for each of the $k$ colors it contains a $t$-subset of that color.
	Let $\{q_1,q_2,\ldots,q_t\}$ be such a $t$-subset. Then $\{q_1,q_2,\ldots,q_t,p_{l-(t'-t)+1},\ldots,p_l\}$ is a $t'$-subset of the same color that is contained in $R$. Therefore, $R$ contains a $t'$-subset of each of the $k$ colors.\qedwhite
\end{proof}

\paragraph{Remark.}	
Note that in the above proof we have only used that the points of $S$ can be ordered as $s_1,\ldots,s_{|S|}$ such that if $\{s_{i_1},\ldots,s_{i_m}\}= S\cap F$ for some $F\in \F$, then also $\{s_{i_1},\ldots,s_{i_{m-1}}\}= S\cap F'$ for some $F'\in \F$.
This implies that similar arguments work not only for $\HH$-regions but for other families as well, especially if we do not insist on $m' = m + t'-t$.

In order to prove Theorem \ref{thm:region}, we reduce the problem of coloring with respect to $\HH$-regions to coloring with respect to axis-parallel boxes (which are possibly higher dimensional). Note that axis-parallel boxes are also $\HH$-regions.

For a point $\mathbf{p} \in \mathbb{R}^d$ we denote by $(\mathbf{p})_i$ the $i$th coordinate of $\mathbf{p}$.
An \emph{axis-parallel} $d$-dimensional (closed) halfspace is a halfspace of the form $\{\mathbf{p} \mid (\mathbf{p})_i \leq \beta\}$ or $\{\mathbf{p} \mid (\mathbf{p})_i \geq \beta\}$ for some $1 \leq i \leq d$ and a number $\beta$.
A $d$-dimensional \emph{box} is the intersection of $d$-dimensional axis-parallel
halfspaces. 
Theorem \ref{thm:region} follows from the next result.

\begin{theorem}\label{thm:box}
	For every integer $t\ge 2$ and positive integers $d,k$ there exists an integer $m=m(k,t,d) \le k^{2^{d-1}}+t-1$ with the following property: the $t$-subsets of every finite point set $S$ in the $d$-dimensional space can be $k$-colored such that every axis-parallel box that contains at least $m$ points contains $t$-subsets of points colored with each of the $k$ colors.
\end{theorem}

Before proving Theorem \ref{thm:box} we show how Theorem~\ref{thm:box} implies Theorem \ref{thm:region}.

\begin{proof}[Proof of Theorem \ref{thm:region} using Theorem \ref{thm:box}]
	Let $\HH=\{H_1,\dots, H_h\}$ be a set of half\-spaces in $\mathbb{R}^d$ such that $H_i= \{\mathbf{x}:\mathbf{x}\in \mathbb{R}^d, \mathbf{A}_i \cdot \mathbf{x} \leq 0\}$ for some $A_i\in \mathbb{R}^d$.
	
	Let $f : \mathbb{R}^d \rightarrow \mathbb{R}^h$ be a mapping such that
	$f(\mathbf{x}) = (\mathbf{A}_1 \cdot \mathbf{x}, \ldots, \mathbf{A}_h \cdot \mathbf{x})$.
	Suppose that $\mathbf{p}$ is a point in $\mathbb{R}^d$ and $H= \{\mathbf{x}:\mathbf{A}_i \cdot \mathbf{x} \leq \beta\}$ is the translate of some halfspace $H_i \in \HH$.
	Then clearly $\mathbf{p} \in H$ if and only if $(f(\mathbf{p}))_i \leq \beta$. 
	
	Let $S$ be a set of points in $\mathbb{R}^d$ and let $R$ be an $\HH$-region that is the intersection of halfspaces $H_{i_1},\ldots, H_{i_l}$ such that $i_j \in \{1,2,\ldots,h\}$ and
	$H_{i_j} = \{\mathbf{x}:\mathbf{A}_{i_j} \cdot \mathbf{x} \leq \beta_j\}$, for $j=1,\ldots,l$.
	It follows that the points of $S$ that belong to $R$
	are exactly the points whose images under $f$ belong to the intersection
	of the set of $h$-dimensional axis-parallel halfspaces $(\mathbf{x})_{i_j} \leq \beta_j$, for $j=1,\ldots,l$.
	
	Since the intersection of $h$-dimensional axis-parallel halfspaces is a $h$-dimen\-sion\-al axis-parallel box, this completes the reduction from coloring with respect to $d$-dimensional $\HH$-regions to coloring with respect to $h$-dimensional axis-parallel boxes. From Theorem \ref{thm:box} we get that $m(k,t,h)=k^{2^{h-1}}+t-1$ is a suitable choice.\qedwhite
\end{proof}

It remains to prove Theorem~\ref{thm:box}.
The following result of Loh~\cite{poshenloh} will be useful. It is a simple application of a classic result known as the Gallai-Hasse-Roy-Vitaver theorem, but can also be proved by a simple application of the pigeonhole principle.
Recall that a \emph{tournament} is a complete oriented graph (between every pair of vertices there is exactly one directed edge).

\begin{theorem}\cite{poshenloh}\label{thm:poshenloh}
	For all positive integers $c$ and $n$, every $c$-coloring of the edges of every $n$-vertex
	tournament contains a monochromatic path with at least $n^{1/c}$ vertices. 
\end{theorem}

The following notation will also be used in the proof of Theorem~\ref{thm:box} 
and later on.
For two points $\p,\q \in \Rd$, we denote by $box(\p,\q)$ the smallest axis-parallel box that contains $\p$ and $\q$.
The \emph{directed type} of a vector is the sequence of signs of its coordinates. 
For example, the directed type of $\vec{v} = (2,1,-3)$ is $(+,+,-)$.

\begin{proof}[Proof of Theorem \ref{thm:box}]	
	It follows from Theorem \ref{thm:t-ind} that it is enough to consider the case $t=2$. That is, given a set of points $S$ in $\Rd$, we wish to color the edges of the complete graph on $S$, such that every box that contains enough points from $S$, contains an edge of each of the $k$ colors.
	
	Every pair of points $\p$ and $\q$ yields two vectors, $\p-\q$ and $\q-\p$, with opposite directed types.
	Among these two vectors pick the one whose first sign is a $+$, 
	and assign its direction and type to the edge $\{\p,\q\}$.
	Denote by $G$ the tournament we obtain by orienting every edge this way (that is, for every edge $\{\p,\q\}$, we direct it from $\p$ to $\q$ if and only if $(\mathbf{p})_1<(\mathbf{q})_1$.)
	
	Suppose that $\{\p,\q\}$ is a pair such that $(\p,\q)$ is an edge (with this direction) in $G$.
	Then the color of this pair and also of the edge $(\p,\q)$, is the minimum of $k$ and the length of a longest path from $\p$ to $\q$ in $G$ such that all the edges along this path have the same type as $(\p,\q)$ (we call such a path a \emph{monotone} path). 
	
	Suppose that $\{\p,\q\}$ gets color $i$ and consider a monotone path of length $i$ from $\p$ to $\q$ in $G$ (assuming that $(\p,\q$) is the edge in $G$). Then for every $1\le j\le i$, if $\r$ is the point at distance $j$ from $\p$ on this path, then the edge $\{\p,\r\}$ must get color $j$. Observe also that all of these edges are in  $box(\p,\q)$. This implies that if a box contains an edge with color $i$, then it contains an edge with color $j$ for every $1\le j\le i$.
	
	Thus, it remains to prove that if a box contains at least $k^{2^{d-1}}+1$ points, then it contains an edge of color $k$. 
	Since there are $2^{d-1}$ different types of directed edges and hence edge colors in $G$, it follows from Theorem \ref{thm:poshenloh} that $G$ contains a monochromatic directed path on $k+1$ vertices. This is by definition a monotone path of length $k$ between its two endpoints, thus the edge between these two endpoints must get color $k$, as required. \qedwhite
\end{proof}

In general the integer $m$ that we get from the proof of Theorem \ref{thm:box} is surely not the best possible. It follows from Theorem \ref{thm:t-ind} that an upper bound for $t=2$ yields a very similar upper bound for any larger $t$. Therefore, we may focus on coloring pairs of points. For the most natural case of axis-parallel rectangles in the plane and coloring pairs of points with two colors (that is, $k=t=d=2$) we determine the optimal upper bound:

\begin{theorem}\label{thm:rectangles}
	The pairs of every finite point set $S$ in the plane can be $2$-colored such that every axis-parallel rectangle that contains at least $3$ points contains pairs of points colored with both colors.
\end{theorem}

\begin{proof}
	As in the proof of Theorem~\ref{thm:box}, for every pair of points $\p$ and $\q$, among the two vectors $\p \rightarrow \q$ and $\q \rightarrow \p$ we pick the one whose first sign is a $+$, 
	and assign its type to the edge $\{\p,\q\}$. Thus, in the plane there are two types of edges: those with the sign sequence $(+,+)$ (i.e., the points are in NE/SW position) and those with the sign sequence $(+,-)$ (i.e., the points are in SE/NW position). 
	We color the pairs of points as follows: color a pair $\{\p,\q\}$ red (resp., blue) if (1)~$\p$ and $\q$ are in NE/SW (resp., SE/NW) position and $|box(\p,\q) \cap S|=2$; or (2)~$\p$ and $\q$ are in SE/NW (resp., NE/SW) position and $|box(\p,\q) \cap S|>2$. 
	
	It is enough to consider an axis-parallel rectangle $R$ such that $|R \cap S|=3$.
	Let $\x,\y,\z$ be the points from $S$ in $R$, ordered by their $x$-coordinates.
	If $\y \in box(\x,\z)$, then the edges $\{\x,\z\}$ and $\{\x,\y\}$ are of different colors.
	Otherwise, if $\y \notin box(\x,\z)$, then the edges $\{\x,\y\}$ and $\{\y,\z\}$ are of different colors. This concludes the proof of the theorem.\qedwhite

\end{proof}

Clearly, the value $3$ in Theorem~\ref{thm:rectangles} is optimal, since if a rectangle contains only $2$ points, then it contains only one pair of points, so it is impossible that it contains pairs of both colors.

We also note that Theorem~\ref{thm:rectangles} could not hold with these constants if instead of coloring all pairs, we would only color Delaunay-edges, as witnessed by the five-point configuration $\{(0,0),(1,1),(2,4),(3,3),(4,2)\}$.

\section{Discussion}

We have introduced the notion of coloring $t$-subsets of vertices of a (geometric) hypergraph such that every hyperedge with enough vertices contains $t$-subsets of different colors (or all colors). This generalizes the well-studied problem of coloring the vertices.
An interesting variant that we focused on is coloring of Delaunay-edges. Perhaps the most interesting problem that we leave open is whether there is a constant $k$ such that it is possible to $k$-color the edges of the Delaunay-graph of a finite point set in the plane with respect to axis-parallel rectangles such that every rectangle that contains $m=m(k)$ points from the point set contains edges of different colors.

For coloring with respect to pseudo-disks we have shown that four colors suffice. It would be interesting to determine whether three or two colors are enough.
It would also be interesting to establish more connections between coloring vertices and coloring Delaunay-edges, and, more generally, to show that a coloring for $t$-subsets of the Delaunay-hypergraph implies a coloring for $t'$-subsets of the Delaunay-hypergraph, for $t'>t$. Theorem \ref{thm:t-ind} is such a statement for the case of coloring all $t$-subsets, yet the proof of Theorem \ref{thm:t-ind} does not seem to be adaptable to the case when we color only the $t$-subsets of the Delaunay-hypergraph.

\smallskip

Finally, let us define a very general notion that contains all the problems considered (including the ones mentioned in the Discussion), namely, \emph{coloring the hyperedges of a hypergraph with respect to another hypergraph and a given relation}.
That is, given a vertex set $V$, two hypergraphs $\HH_1$ and $\HH_2$ on this vertex set, and a relation $R$ (e.g., intersection, containment, reverse containment),
we define the hypergraph $\HH(\HH_1,\HH_2,R)$ as follows. Its vertex set is the hyperedge set of $\HH_1$ and for every hyperedge $H_2$ in $\HH_2$ it has a hyperedge which contains exactly those hyperedges of $\HH_1$ which are in relation $R$ with $H_2$.
Now a (proper) coloring of the hyperedges of $\HH_1$ with respect to $\HH_2$ and $R$ means to color the vertices of $\HH(\HH_1,\HH_2,R)$ (that is, the hyperedges of $\HH_1$) such that every (large enough) hyperedge of $\HH(\HH_1,\HH_2,R)$ (defined by some $H_2\in \HH_2$) is not monochromatic.

We conclude by saying how all the problems considered are special cases of this very general definition.
For geometric problems the vertex set is always a finite point set $S$ and we are given a family of regions $\F$. We say that a hyperedge is defined by some region if it is the intersection of the region with $S$.

For the point coloring problem $\HH_1$ consists of the one-element hyperedges while $\HH_2$ consists of the hyperedges defined by the regions, while $R$ is the containment relation. For the dual problem of coloring regions, $\HH_2$ consists of the one-element hyperedges while $\HH_1$ consists of the hyperedges defined by the regions, while $R$ is the reverse containment relation. For the intersection hypergraph coloring problem $\HH_1=\HH_2$ consists of the hyperedges defined by the regions and $R$ is the intersection relation.
Considering the problem of coloring all $t$-subsets with respect to some regions, $\HH_1$ consists of all $t$-subsets, $\HH_2$ consists of the hyperedges defined by the regions (and $R$ is the containment relation). For $t$-subsets that are themselves hyperedges $\HH_2$ consists of the hyperedges defined by the regions and $\HH_1$ consists of the hyperedges of $\HH_2$ of size $t$ (and $R$ is the containment relation).
Note that the respective hypergraph coloring problems are also special cases, where instead of hyperedges defined by some region we have hyperedges of some given hypergraph $\HH$.

\subsubsection*{Acknowledgment}

We would like to thank our anonymous referees for several suggestions that improved the presentation of our results.

\end{document}